\documentclass[11pt]{amsart}

\usepackage{amscd,amsmath,latexsym,amsthm,amsfonts,amssymb,graphicx,geometry}

\usepackage{enumerate}

\usepackage[usenames,dvipsnames]{color}

\usepackage{hyperref}
\hypersetup{
colorlinks=true,
linkcolor=blue,
citecolor=cyan
}

\usepackage{geometry}
\allowdisplaybreaks

\usepackage[norefs,nocites]{refcheck}

\newtheorem{theorem}{Theorem}[section]
\newtheorem{proposition}[theorem]{Proposition}

\newtheorem{corollary}[theorem]{Corollary}

\newtheorem{lemma}[theorem]{Lemma}

\numberwithin{equation}{section}

\newcommand{\cE}{\mathcal{E}}
\newcommand{\bi}{\mathbf{i}}

\newcommand{\br}{\mathbf{r}}
\newcommand{\bs}{\mathbf{s}}
\newcommand{\bt}{\mathbf{t}}
\newcommand{\bu}{\mathbf{u}}

\newcommand{\fc}{\mathfrak{c}}

\begin{document}

\title[Large structures within the class of summing operators]{Large structures within the class of summing operators}


\author[Albuquerque]{N. G. Albuquerque\textsuperscript{*}}
\address[N. G. Albuquerque]{Departamento de Matem\'{a}tica \newline\indent
Universidade Federal da Para\'{i}ba \newline\indent
Jo\~ao Pessoa - PB \newline\indent
58.051-900 (Brazil)}
\email{nacib.albuquerque@academico.ufpb.br}

\author[Coleta]{L. Coleta\textsuperscript{**}}
\address[L. Coleta]{Departamento de Matem\'{a}tica \newline\indent
Universidade Federal da Para\'{i}ba \newline\indent
Jo\~ao Pessoa - PB \newline\indent
58.051-900 (Brazil)}
\email{lindinescoleta@gmail.com}

\thanks{\textsuperscript{*} Supported by CNPq Grant 312167/2021-01 and Grant 2019/0014 Para\'iba State Research Foundation (FAPESQ)}
\thanks{\textsuperscript{**} Supported by CAPES}



\begin{abstract}
We investigate lineability/spaceability problems within the setting of multilinear summing operators on quasi-Banach sequence spaces. Furthermore, we deal with the contemporary geometric notions of pointwise-lineability and $(\alpha,\beta)$-lineability. Among other results, we prove that the class of multilinear operators taking values on $\ell_q \, (0<q \leq \infty)$ that are absolutely but not multiple summing is pointwise $\fc$-spaceable when non-empty. Spaceability in other classes, such as Dunford-Pettis operators and multilinear operators on general summable scalar families, is also studied.
\end{abstract}

\maketitle

\section{Introduction}

The modern terminology of \emph{lineability} and \emph{spaceability} was coined by Gurariy, Aron, and Seoane in the seminal paper \cite{Aron}: for a cardinal $\alpha$, a subset $A$ of a vector space $X$ is said to be $\alpha$-\emph{lineable} if  $A\cup\left\{  0\right\}  $ contains a $\alpha$-dimensional linear subspace of $X$. If $X$ is, in addition,  a topological vector space, then  $A$ is called $\alpha$-\emph{spaceable} if $A\cup\left\{  0\right\}  $ contains a closed  $\alpha$-dimensional linear subspace of $X$. These concepts are widely investigated in several research branches. For an overview, we refer the reader to recent works in \cite{bernal-jfa,BerPS}, and, to our knowledge, the most comprehensive panorama of the matter are given in \cite{lineability-book}, and \cite{BerPS}. The search for new lineability notions yields to several new concepts, and some recent ones can be found in \cite{favaro} and \cite{pr_pointw}.


Many authors contributed to investigating lineability aspects of certain classes of summing operators. Puglisi and Seoane proved in \cite{puglisi} that, under certain restrictions over a Banach space $E$, the set $\mathcal{L}(E,\ell_2)\backslash \Pi_1(E,\ell_2)$ of bounded linear non-absolutely summing operators is lineable. Later in \cite{Pellegrino1}, Botelho, Diniz and Pellegrino obtained the lineability of $\mathcal{K}(E;F) \setminus \Pi_p(E;F)$, the set of compact but not $p$-summing operators, with $p \geq 1$, with some conditions of reflexivity, infinite dimension, and unconditional basis on the Banach spaces $E,F$. More generally, Kitson and Timoney in \cite{KT} obtained a remarkable and general spaceability result for Fr\'echet spaces with several applications, among them, the spaceability of $\mathcal{K}(E,F)\backslash \bigcup_{1\leq p<\infty}\Pi_p(E,F)$, whenever $E$ is superreflexive and $F$ is infinite dimensional. Hernand\'ez, Ruiz, and S\'anchez took a step further and dealt with the spaceability of operators ideals: we refer to \cite{sanchez_jmaa} for the spaceability of $\mathcal{I}_1(E;F) \setminus \mathcal{I}_2(E;F)$, with $\mathcal{I}_1,\mathcal{I}_2$ operator ideals over certain Banach spaces $E,F$. Alves and Turco dealt with the spaceability of sets of $p$-compact linear maps where the domain and codomain are $\ell_r$, with $1\leq r < \infty$, and $c_0$ was investigated in depth and many results are obtained in \cite{at-jmaa}. In the context of quasi-Banach spaces, in \cite{DanielT} Tom\'az proved that $\mathcal{L}(\ell_p,\ell_p) \setminus  \bigcup_{1\leq s\leq r<\infty} \Pi_{(r,s)} (\ell_p,\ell_p)$ is maximal lineable for any $0<p<1$. In the multilinear environment, G. Ara\'ujo and D. Pellegrino established in \cite{araujo} that set $\mathcal{L}(^m\ell_p, \mathbb{K}) \setminus \Pi_{(r;s)}^{\textrm{ms}} (^m\ell_p, \mathbb{K})$ of bounded $m$-linear but not multiple summing forms is maximal spaceable, for $m\geq 2,\, p\in[2,\infty),\, 1\leq s<p*$ and $r < \frac{2ms}{s+2m-ms}$. Here, as usual, $\Pi^\textrm{as}$ and $\Pi^\textrm{ms}$ stand for the class of \emph{absolutely} and \emph{multiple} multilinear summing classes, respectively. Also in \cite{FPP} the authors investigated the lineability of non-bounded and non-absolutely summing multilinear operators.

It is well known from the classical Bohnenblust-Hille inequality and also the Defant-Voigt theorem that $\Pi_{(r;1)}^{\textrm{ms}} \left(^m E; \mathbb{K}\right) \subsetneq \mathcal{L}\left(^m E; \mathbb{K} \right) = \Pi_{(r;1)}^{\textrm{as}} \left(^m E; \mathbb{K}\right)$ for all $1 \leq r < \frac{2m}{m+1}$. Therefore $\Pi_{(r;1)}^{\textrm{as}} \left(^m E; \mathbb{K}\right) \setminus \Pi_{(r;1)}^{\textrm{ms}} \left(^m E; \mathbb{K}\right)$ is non-empty and the search for large closed vector spaces naturally arises in this setting. Similar problems to this one are answered as an application of the result we provide. A particular case we prove is that the class of multilinear operators taking values on sequence spaces that are absolutely but not multiple summing, for instance, for $E_1,\dots,E_m$ Banach spaces, $1\leq r \leq s < \infty$ and $q \in (0,\infty]$,
\[
\Pi_{(r,s)}^{\textrm{as}} (E_1,\dots,E_m;\ell_q)
\setminus
\Pi_{(r,s)}^{\textrm{ms}} (E_1,\dots,E_m;\ell_q),
\]
is $\mathfrak{c}$-spaceable when it is not empty. In fact, we prove a strong version dealing with the restrictive pointwise-lineability/spaceability notion and in the general setting of $\Lambda$-multiple summing class (see   \cite{aacnnpr-afa,bpr-collect,botelho-blocks,popa-lma} for more details).

The main goals of this paper are to contribute to the research of lineability and spaceability on the class of multilinear operators in two perspectives not often covered: we investigate the geometric concepts of pointwise and $\left(\alpha,\beta\right)$-lineability / spaceability (we discuss it with further details in Sections \ref{sec-prelim} and \ref{sec-tuples}), which are more restrictive than the classical notion; also we deal with operators taking value on quasi-Banach sequence spaces. The panorama is quite different from Banach spaces when dealing with quasi-Banach spaces, such as non-locally spaces $\ell_q$ for $0<q<1$. Thus the search for lineability/spaceability techniques in quasi-Banach spaces can be an exciting and challenging matter.

This paper is organized as follows. Section \ref{sec-prelim} presents preliminary concepts and notations used throughout the paper. In Section \ref{sec-summ-op} we obtain the main result of the paper: the pointwise $\mathfrak{c}$-spaceability of classes of $\Lambda$-summing operators (Theorem \ref{th-gen}) and some applications are presented. Section \ref{sec-dp} deals with absolutely summing operators that fail to be Dunford-Petis (or completely continuous). In Section \ref{sec-tuples} we prove that some sets of operators taking values on spaces of general $q$-summable scalar families  $\ell_q \left(I \right)$ can be $(\alpha,\text{card\,} I)$-lineable, for $I$ some infinite set and some cardinal $\alpha<\text{card\,} I$.

\section{Background and preliminaries} \label{sec-prelim}

Throughout this paper we deal with vector spaces over the scalar field $\mathbb{K}$, which can be either $\mathbb{R}$ or $\mathbb{C}$, and we write $\text{card\,} \mathbb{R} = \mathfrak{c}$ and $\text{card\,} \mathbb{N} = \aleph_0$. Banach spaces over $\mathbb{K}$ shall be denoted by capital letter (eventually with indexes) such as $E,E_{1},\dots,$ $E_{m},F,G,H$, unless stated otherwise. The topological dual and the closed unit ball of $E$ will be denoted by $E^{\prime}$ and $B_{E}$, respectively. We will denote by $\mathcal{L}\left(E_1,\dots,E_m;F\right)$ the Banach space  of bounded $m$-linear operators from $E_1 \times \cdots \times E_m$ to $F$ endowed with the usual sup norm.

A \emph{quasi-norm} on a vector space $X$ is a non-negative real-valued function $\|\cdot\|:X \to [0,\infty)$ satisfying for all $x,y \in X$ and $\lambda \in \mathbb{K}$ the following properties: $\|x\|=0$ only if $x=0$; $\|\lambda x\|=|\lambda|\|x\|$; and $\|x+y\|\leq C(\|x\|+\|y\|)$ with $C\geq 1$ an universal constant. If, in addition, the map $\|\cdot\|$ is \emph{$p$-subadditive} for some $0< p \leq 1$, i.e., it satisfies $\|x+y\|^p\leq \|x\|^p+\|y\|^p$ for all $x,y \in X$, $\|\cdot\|$ is called a \textit{$p$-norm}. Clearly we have a norm when $C=1$ or $p=1$. A quasi-norm defines a metrizable vector topology on $X$ whose base of neighborhoods of the origin is given by sets of the form $\{x \in X : \|x\|<1/n\},\, n \in \mathbb{N}$. $X$ is called a \emph{quasi-Banach} space when it is complete for this metric. A $p$-subadditive quasi-norm induces a metric topology on $X$ given by $d(x,y) := \|x-y\|^p$. A quasi-Banach space with an associated $p$-norm is also called a \emph{$p$-Banach space}. A deep result known as the \emph{Aoki-Rolewicz} theorem (see \cite{book1,pietsch-book}) guarantees that every quasi-normed space is $p$-normable for some $0<p\leq1$, i.e., the space can be endowed with an equivalent quasi-norm which is a $p$-norm. Thus we assume that a quasi-Banach space is $p$-Banach for some $0 < p \leq 1$. Many definitions and classical results in Banach spaces can be translated in a natural way to quasi-Banach ($p$-Banach) spaces, like standard results depending on Baire Category Theorem as Open Mapping Theorem. 
Nevertheless, $p$-normed spaces are not necessarily locally convex. 
Results which hold on Banach spaces and are based somehow on the local convexity (\emph{e.g.}, Hahn-Banach extension property or the Krein-Milman theorem) are no longer valid, in general, in that setting. The extension of lineability/spaceability arguments from Banach to quasi-Banach spaces is not straightforward. Thus the search for large closed infinite dimensional subspaces of $p$-Banach spaces is quite a delicate issue. Therefore, looking for lineability/spaceability techniques that also cover this environment spaces seems interesting. For more details on quasi-Banach and $p$-Banach spaces, we refer to \cite{book1}.


Given a positive integer $m$, $\bi := (i_1,\dots,i_m)$ stands for a multi-index in $\mathbb{N}^m$, and $\br := (r_1,\ldots,r_m) \in [1,+\infty]^m$ a multi-parameter. By $\ell_\br(E)$ we denote the Banach space that gathers all $E$-valued multi-matrices $\left( x_\bi \right)_{\bi \in \mathbb{N}^m} \in E^{\mathbb{N}^m}$ with finite $\ell_\br$-norm, that is, when $\br \in [1,+\infty)^m$ we have
\[
\left\| \left( x_\bi \right)_{\bi \in \mathbb{N}^m} \right\|_{\ell_\br(E)}
:=
\left( \sum_{i_1=1}^{\infty}
    \left( \cdots
      \left( \sum_{i_m=1}^{\infty} 
      \left\|  x_\bi \right\|_{E}^{r_m}
      \right)^{\frac{r_{m-1}}{r_m}} \cdots
  \right)^{\frac{r_1}{r_2}}
\right)^{\frac{1}{r_1}} < \infty.
\]
We simply write $\ell_\br = \ell_\br (\mathbb{K})$ and when no precision is needed $\| \cdot \|_{\ell_\br(E)} =  \| \cdot \|_\br$. Notice that $\ell_\br(E)$ is a quasi-Banach space when some $0<r_j<1$. Given a non-empty set of indexes $\Lambda \subset \mathbb{N}^{m}$, by $\left( x_\bi \right)_{\bi \in \Lambda}$ we mean an $E$-valued multi-matrix indexed over $\Lambda$, and its $\ell_\br$-norm can be seen as 
\[
\left\| \left( x_\bi \right)_{\bi \in \Lambda} \right\|_{\ell_\br(E)}
=
\left\| \left( x_\bi \cdot 1_\Lambda (\bi) \right)_{\bi \in \mathbb{N}^m} \right\|_{\ell_\br (E)}
\]
where $1_\Lambda$ is the characteristic function of $\Lambda$. It is worth pointing out some simple but useful monotonicity properties: $ \| \cdot \|_s \leq \| \cdot \|_r $ whenever $0<r<s<\infty$; $ \left\| \left( \alpha_\bi \right)_{\bi \in \Lambda}  \right\|_\br \leq \left\| \left( \alpha_\bi \right)_{\bi \in \Gamma} \right\|_\br$ for all scalar multi-matrix $\left( \alpha_\bi \right)_{\bi \in \mathbb{N}^m}$ and for (non-empty) indexes sets $\Lambda \subset \Gamma \subset \mathbb{N}^{m}$; $\| (\alpha_\bi)_{\bi \in \Lambda} \|_\br \leq \| (\beta_\bi)_{\bi \in \Lambda} \|_\br$ if $0 \leq \alpha_\bi \leq \beta_\bi$ for all $\bi \in \Lambda$.

A \emph{quasi-Banach standard sequence space} over a Banach space $X$ is an infinite dimensional quasi-Banach space $\mathcal{E}$ whose elements are $X$-valued sequences with the usual operations enjoying the following conditions:
\begin{enumerate}[(i)]
\item There is a constant $C>0$ such that
\[
\|x_j\|_{X}\leq C\|x\|_{\mathcal{E}},
\]
for every $x=(x_j)_{j\in\mathbb{N}}\in \mathcal{E}$ and all $j\in \mathbb{N}$.

\item If $x=(x_j)_{j\in\mathbb{N}}\in \mathcal{E}$ and $(x_{n_k})_{k\in\mathbb{N}}$ is a subsequence of $x$, then $(x_{n_k})_{k\in\mathbb{N}}\in \mathcal{E}$ and
\[
\|(x_{n_k})_{k\in\mathbb{N}}\|_{\mathcal{E}}\leq \|x\|_{\mathcal{E}}.
\]

\item If $x = (x_j)_{j\in\mathbb{N}} \in \mathcal{E}$ and $\mathbb{N}':=\{n_1<n_2<n_3<\cdots\}$ is an infinite subset of $\mathbb{N}$, then the $X-$valued sequence $y=(y_j)_{j\in\mathbb{N}}$ defined as
\begin{equation*}
y_j := 
\begin{cases}
x_i, & \mbox{ if } j=n_i,\\
0, & \mbox{ otherwise },
\end{cases} 
\end{equation*}
belongs to $\mathcal{E}$ and
$
\|y\|_{\mathcal{E}} \leq \| x \|_{\mathcal{E}}.
$
\end{enumerate}

It is plain that combining (ii) and (iii) yields $\|y\|_\cE = \|x\|_\cE$. The $n$-th coordinate of a sequence $z \in \mathcal{E} \subset X^{\mathbb{N}}$ will be denoted by $z_n$ or, when more precision is needed, by $z(n)$. Unless the contrary is explicitly established, $\mathcal{E}$ will denote a quasi-Banach standard sequence space. Usual sequence spaces are (quasi)-Banach standard sequence spaces as, for instance, the classical sequence spaces $\ell_p(E), \ell_p^w(E), \ell_p^u(E)$ with $p \in (0,\infty)$, $c_0(E),\, c (E)$, and the Lorentz space $\ell_{p,q}$. For a deeper discussion on general sequence spaces, more details and examples, we refer the reader to \cite{bf-seqspaces,dfpr_racsam2020}. The proof of the next lemma is straightforward.


\begin{lemma} \label{lemma-norm-infty}
Let $\mathbf{r} := (r_1,\dots,r_m) \in \left[1, \infty \right]^m$ and $\Lambda \subset \mathbb{N}^m$. If $\left( \alpha_{\bi} \right)_{\mathbf i \in \Lambda} \notin \ell_{\br} (\mathcal{E})$, then either
\[
\left( \left( \alpha_{\mathbf i} (j) \right)_{j \in \mathbb{N}'} \right)_{\mathbf i \in \Lambda } \notin \ell_{\mathbf r} (\mathcal{E})
\quad \text{ or } \quad
\left( \left( \alpha_{\mathbf i} (j) \right)_{j \in \mathbb{N} \setminus \mathbb{N}'} \right)_{\mathbf i \in \Lambda } \notin \ell_{\mathbf r} (\mathcal{E}),
\]
for any $\mathbb{N}' \subset \mathbb{N}$ with $ \text{card\,} \mathbb{N}' = \text{card} ( \mathbb{N} \setminus \mathbb{N}') = \aleph_0$.
\end{lemma}

Recently, Pellegrino and Raposo Jr. in \cite{pr_pointw} introduced the \emph{geometric} notion of pointwise-lineability. Let $V$ be a vector space and $\alpha$ be a cardinal number with $\alpha \leq \dim V$. A set $A\subset V$ is \emph{pointwise} $\alpha$-lineable if for each $x \in A$ there is a subspace $W = W_{x,\alpha}$ such that $x \in W \subset A \cup \{0\}$ and $\dim W = \alpha$. If $X$ is a topological vector space and $W$ can be chosen to be closed, $A$ is said \emph{pointwise} $\alpha$-spaceable. It is common to refer to $A$ as \emph{lineable} or \emph{spaceable} when $\alpha \geq \aleph_0$. Pointwise lineability is a more restrictive concept than classical lineability. The latter faces a technicality not often covered in usual lineability results: the subspace generated contains every initial mother vector. Also, this notion is related with $\left(\alpha,\beta\right)$-lineability (see Section \ref{sec-tuples}). For more details, we refer to \cite{pr_pointw}.

We finish this section with the following result, which will reveal itself useful and it is also closely connected to a lineability technique that dates back to \cite{pt-bbms-2009}. First we recall a basic ideal property of a quasi-Banach multi-ideal operators \( \left(\mathcal{M}, \|\cdot\|_{\mathcal{M}} \right)\): if \(S \in \mathcal{M}(E_1,\dots,E_m; F),\, u_j \in \mathcal{L}(G_j;Ej),\, j=1,\dots,m,\, \) and \(t \in \mathcal{L}(F;H)\), then \(t \circ S \circ (u_1,\dots,u_j) \in \mathcal{M}(G_1, \dots, G_m; H) \) and \( \| t \circ S \circ (u_1,\dots,u_j) \|_{\mathcal{M}} \leq \|t\| \cdot \|S\|_{\mathcal{M}} \cdot \|u_1\| \cdots \|u_m\| \) (see \cite{ideals_botelho} and references therein for more details). Also, we say that $\left( \mathbb{N}_k \right)_{k \in \mathbb{N}}$ is a \emph{countably disjoint decomposition} of the natural numbers when we can write $\mathbb{N} = \bigcup_{k \in \mathbb{N}} \mathbb{N}_k$ with each $\mathbb{N}_k := \left\{ n_1^{(k)} < n_{2}^{(k)} < \cdots \right\}$ an infinite subset of $\mathbb{N}$ and $\mathbb{N}_i \cap \mathbb{N}_j = \emptyset$ whenever $i\neq j$.

\begin{proposition}\label{prop_pw-mothervector}
Let $\mathcal{M}$ be quasi-Banach operators multi-ideals. Given a countably disjoint decomposition $\left( \mathbb{N}_k \right)_{k \in \mathbb{N}}$ of $\mathbb{N}$ and a non-trivial operator $T \in \mathcal{M} \left(E_1,\dots,E_m;\cE\right)$, we have the following.
\begin{enumerate}[(i)]
\item For each natural $k \geq 1$, the map $T_k : E_1 \times \cdots \times E_m \to  \mathcal{E}$ defined by
\begin{equation*}
T_k x \left( j \right) := 
\begin{cases}
0, & \mbox{ if } j \notin \mathbb{N}_k,\\
Tx (i), & \mbox{ if } j = n_i^{(k)} \in \mathbb{N}_k,
\end{cases} 
\end{equation*}
lies in $\mathcal{M} \left(E_1,\dots,E_m;\cE\right)$, and $\left\| T_{k} x \right\|_{\cE} =\left\| Tx \right\|_{\cE}$ for all $x\in E_1 \times \cdots \times E_m$. Moreover, $\left\{T_k : k \geq 2 \right\} \cup \{T\}$ is linearly independent.

\item For some $0<p\leq1$, such that the map $\Psi :\ell_p \to  \mathcal{M} \left(E_1,\dots,E_m;\cE\right)$ given by
\begin{equation*} 
\Psi a :=\alpha_1 T + \sum_{j=2}^{\infty} \alpha_j T_j,  \quad  a=(\alpha_j)_{j\in \mathbb{N}} \in \ell_p,
\end{equation*}
is well-defined, linear, bounded, and injective. 
\end{enumerate}
\end{proposition}

\begin{proof}
\noindent (i) Since $T$ is a non-zero operator, there exist $x_0 \in E_1 \times \cdots \times E_m$ and $j_0 \in \mathbb{N}$ such that
$T x_0 (j_0) \neq 0,$
where $T x_0 (j_0)$ is the $j_0$-th coordinate of the sequence $Tx_0 \in \mathcal{E}$. For simplicity, we shall assume that $j_0 \in \mathbb{N}_1$. Using the notation in condition (iii) of standard sequence spaces, for each $k\in \mathbb{N}$ we take $\mathbb{N}' = \mathbb{N}_k = \left\{ n_1^{(k)} < n_{2}^{(k)} < \cdots \right\}$, and we define $V_k : \mathcal{E} \to \mathcal{E}$ by $V_k x := (y_j)_{j \in \mathbb{N}}$ for $x = (x_n)_{n\in \mathbb{N}} \in \mathcal{E}$. It is clear that $V_k$ is a bounded linear map with $\|V_k\| \leq 1$, $T_k = V_k \circ T \in \mathcal{M} \left(E_1,\dots,E_m;\cE\right)$, and $\|T_k\|_{\mathcal{M}} \leq \|T\|_{\mathcal{M}}$. Using conditions (ii) and (iii) of standard sequence spaces, and the fact that $\{T_k x\}_{k \in \mathbb{N}}$ is a family of sequences in $\mathcal{E}$ whose supports are pairwise disjoint for any $x$, it is straightforward to conclude the remaining statements.

\vspace*{3mm}

\noindent (ii) $A := \mathcal{M} \left(E_1,\dots,E_m;\cE\right) \subset \mathcal{L} \left(E_1,\dots,E_m;\cE\right)$ is a $p$-Banach for some $p \in (0,1]$. For any $a = (\alpha_j)_{j \in \mathbb{N}} \in \ell_p$,
\[
\sum_{j \geq 2} \|\alpha_j T_j\|_A^p
=\sum_{j \geq 2} |\alpha_j|^p \|T_j\|_A^p
=\sum_{j \geq 2} |\alpha_j|^p \|T\|_A^p
=\|T\|_A^p \sum_{j \geq 2}|\alpha_j|^p
<\infty.
\]
Hence the series $\sum_{j \geq 2} \alpha_j T_j$ converges in $A$ and, consequently, the map $\Psi :\ell_p \to A$ is well defined and linear. In order to simplify the notation we write $\Psi_a := \Psi a$ for $a \in \ell_p$. Now suppose that $\Psi_a =0$. Combining $T x_0 (j_0) \neq 0$ with the fact that the sequence $T_j x$ has null coordinates on $\mathbb{N} \setminus \mathbb{N}_j$, for any $x \in E_1 \times \cdots \times E_m$ and $ j\geq 2$,
\[
0 = \Psi_a x_0 (j_0)
= \alpha_1 T x_0 (j_0) + \sum_{j=2}^{\infty} \alpha_j T_j x_0 (j_0) = \alpha_1 T x_0 (j_0),
\]
implies that $\alpha_1=0$. Since $(T_j)_j$ is linearly independent, we have $\alpha_j=0$ for all $j$. Therefore $a=0$, and this concludes that $\Psi$ is injective.

\end{proof}

\section{Spaceability of multilinear summing operators} \label{sec-summ-op}

In this Section the search of large topological structures is focused in a general environment of multilinear summing operators. The concept of $\Lambda$-multiple summing (or $\Lambda$-summing) is a notion of multilinear summing operators that encompasses the classical notions of absolutely and multiple multilinear summing operators. It was independently introduced in \cite{bpr-collect,popa-lma} and recent developments in this environment can be found, e.g., in \cite{aacnnpr-afa,botelho-blocks}. More precisely, let $m \in \mathbb{N}$ be a positive integer, $\br,\, \bs \in [1,+\infty)^m$ and $\Lambda \subset \mathbb{N}^{m}$ a set of indexes. An $m$-linear operator $T: E_1 \times \cdots \times E_m \to F$ is $\Lambda$-$(\mathbf{r}; \mathbf{s})$-summing if there is a constant $C>0$ such that
\begin{equation}  \label{lambda-def}
\left\| \left( T x_\bi \right)_{\bi \in \Lambda}  \right\|_{\ell_{\br} \left( F \right)}
\leq C \prod_{k=1}^{m} \sup_{\phi_k \in B_{E_k^{\prime}}} \left(\sum_{i=1}^{N} \left| \phi(x_i^{(k)}) \right|^{s_k} \right)^{\frac{1}{s_k}}
\end{equation}
for all $N \in \mathbb{N}$ and $x_{i}^{(k)} \in E_k,\, k=1,\dots,m,\, i=1,\dots,N$. For simplicity of notation we write $Tx_\bi := T\left( x_{i_1}^{(1)},\dots, x_{i_m}^{(m)} \right)$. The class of all operators that fulfills the previous inequality is denoted by $\Pi_{(\mathbf{r};\mathbf{s})}^{\Lambda} \left( E_1,\dots,E_m;F \right)$, which is a Banach space endowed with the norm $\pi_{(\mathbf{r};\mathbf{s})}^\Lambda(T)$ taken as the infimum of the constants $C>0$ satisfying \eqref{lambda-def}. Notice that, by taking $\Lambda = \text{Diag}\left( \mathbb{N}^m \right) := \left\{ \left( n,\cdots,n \right) \in \mathbb{N}^m : n \in \mathbb{N} \right\}$ and $\Lambda = \mathbb{N}^{m}$, the $\Lambda$-summing class $\Pi^\Lambda$ recovers both \emph{absolutely} and \emph{multiple} multilinear summing classes, respectively
It is noteworthy to mention that given $\Lambda \subset \Gamma \subset \mathbb{N}^{m}$, the following inclusions and norm inequalities are easily seen (over fixed parameters $(\br;\bs)$ we omit): $\Pi^{\textrm{ms}} \subset \Pi^\Gamma \subset \Pi^\Lambda \subset \Pi^{\textrm{as}}$ and $\pi^{\textrm{ms}}(\cdot) \leq \pi^\Gamma(\cdot) \leq \pi^\Lambda(\cdot) \leq \pi^{\textrm{as}} (\cdot)$. Moreover, one can deal with the codomain $F$ as a quasi-Banach space and, with standard reasoning, $\Pi_{(\br;\bs)}^{\Lambda} \left(E_1,\dots,E_m;F\right)$ also is quasi-Banach. For more details and applications in the theory of absolutely multiple summing operators we refer the reader to \cite{botelho-blocks} and the references therein.

\vspace*{3mm}

The main result of this section reads as follows.

\begin{theorem} \label{th-gen}
Let $\br,\bs,\bt,\bu \in [1,+\infty)^m$ and $\Lambda \subset \Gamma \subset \mathbb{N}^m$ sets of indexes. Then
\[
\Pi_{(\br, \bs)}^{\Lambda}  \left( E_1,\ldots,E_m; \mathcal{E} \right) \setminus \Pi_{(\bt, \bu)}^{\Gamma} \left( E_1,\ldots,E_m; \mathcal{E} \right)
\]
is either empty or pointwise $\mathfrak c$-spaceable.
\end{theorem}

\begin{proof} For simplicity, we write $A := \prod_{(\mathbf{r},\mathbf{s})}^{\Lambda}( E_1,\ldots,E_m; \mathcal{E})$ and $B := \prod_{(\bt, \bu)}^{\Gamma}( E_1,\ldots,E_m; \mathcal{E})$. Let us suppose there exists an operator $T \in A\setminus B$. We prove first that $A \setminus B$ is pointwise $\mathfrak{c}$-lineable.  Let us fix $x_0 \in E_1 \times \cdots \times E_m$ and $j_0 \in \mathbb{N}$ such that $T x_0 (j_0) \neq 0$. We also consider weakly summable sequences $z^{(j)} = \left(z_{n}^{(j)} \right)_{n \in \mathbb{N}} \in \ell_{u_j}^w(E_j), \, j=1,\ldots,m$, such that
\[
\left\| \left( T z_\mathbf{i} \right)_{\mathbf{i} \in \Gamma}  \right\|_{\ell_{\mathbf{t}} \left( \mathcal{E} \right)} = \infty.
\]
Recall the notation $T z_\bi := T\left( z_{i_1}^{(1)},\dots, z_{i_m}^{(m)} \right)$. Lemma \ref{lemma-norm-infty} assures that we can take $\mathbb{N}_1 \subset \mathbb{N}$ with $j_0 \in \mathbb{N}_1,\, \text{card\,} \mathbb{N}_1 = \text{card} ( \mathbb{N} \setminus \mathbb{N}_1) = \aleph_0$ and also such that 
\begin{equation} \label{N1-norm-infty}
\left\| \left( \left(T z_\mathbf{i} (n) \right)_{n \in \mathbb{N}_1} \right)_{\mathbf{i} \in \Gamma}  \right\|_{\ell_{\mathbf{t}} \left( \mathcal{E} \right)} = \infty.
\end{equation}

Now we take the following countably infinite decomposition: $\mathbb{N} \setminus \mathbb{N}_1 = \bigcup_{k \geq 2} \mathbb{N}_k$. Then Proposition \ref{prop_pw-mothervector} (i), applied with \(\mathcal{M} =  \Pi_{(\br, \bs)}^{\Lambda} \), provides a sequence of (linearly independent) operators $(T_k)_{k\in\mathbb{N}}$ with $\left\| T_{k} x \right\|_{\cE} =\left\| Tx \right\|_{\cE}$ for all $x\in E_1 \times \cdots \times E_m$. Clearly this implies $\|T_k\|_A = \|T\|_A$ for all $k$ and also $(T_k)_{k\in\mathbb{N}}$ belongs to $A \setminus B$.

In order to obtain the $\fc$-dimensional space in $A \setminus B$ that contains $T$, we use the construction and notation of Proposition \ref{prop_pw-mothervector} (ii): for some $0<p \leq 1$, the map $\Psi :\ell_p \to A$ is well-defined, bounded, linear and injective. According to this, $\Psi (\ell_p) \subset A$ is a $\fc$-dimensional space that contains $T$. We shall now prove that $\Psi (\ell_p) \setminus \{0\} \subset A \setminus B$, that is: $\Psi_a \notin B$, for any $a = (\alpha_j)_j \in \ell_p \setminus \{0\}$. Consider $z^{(j)} \in \ell_{u_j}^w(E_j),\, j=1,\ldots,m$, as in \eqref{N1-norm-infty}. If $\alpha_1 \neq 0$, using again that every image of $T_j$ is a sequence with null coordinates on $\mathbb{N} \setminus \mathbb{N}_j$ for $j \geq 2$,
\begin{align*}
\left\| \Psi_a  z_\bi \right\|_{\cE}
& =\left\| \left( \alpha_1 T z_\bi (n) + \sum_{j\geq2} \alpha_j T_j z_\bi (n) \right)_{n\in\mathbb{N}}\right\|_{\cE}\\
& \geq \left\| \left( \alpha_1 T z_\bi (n) + \sum_{j\geq2} \alpha_j T_j z_\bi (n) \right)_{n \in \mathbb{N}_1}\right\|_{\cE}\\
& =\left\| \left(\alpha_1 T z_\bi (n) \right)_{n\in\mathbb{N}_1} \right\|_{\cE}\\
& = |\alpha_1| \left\| \left( T z_\bi (n) \right)_{n \in \mathbb{N}_1}  \right\|_{\cE}.
\end{align*}
The monotonicity of the quasi-norm yields
\[
\left\| \left( \Psi_a z_\bi \right)_{\mathbf{i} \in \Gamma}  \right\|_{\ell_{\mathbf{t}} \left( \cE \right)}
\geq |\alpha_1| \left\| \left( \left(T z_\mathbf{i} (n) \right)_{n \in \mathbb{N}_1} \right)_{\mathbf{i} \in \Gamma}  \right\|_{\ell_{\bt} \left( \cE \right)} = \infty,
\]
with the last inequality being a consequence of \eqref{N1-norm-infty}, which gives $\Psi_a \notin B$. We now turn to the case $\alpha_1 = 0$, in which we proceed in a similar fashion. Let us consider $\alpha_i\neq 0$ for some $i\in \mathbb{N}$. Using the definition of the operators $(T_k)_{k \in \mathbb{N}}$ in Proposition \ref{prop_pw-mothervector} (i),
\[
\left\| \Psi_a z_\bi \right\|_{\cE}
= \left\| \left(\sum_{j\geq2} \alpha_j T_j z_\bi (n)\right)_{n \in \mathbb{N}} \right\|_{\cE}
\geq  \left\| \left( \alpha_i T_i z_\bi \left(n_k^{(i)}\right)\right)_{k \in \mathbb{N}} \right\|_{\cE}
= |\alpha_i|\left\| \left( T z_\bi \left(k\right)\right)_{k \in \mathbb{N}} \right\|_{\cE},
\]
thus
\[
\left\| \left( \Psi_a  z_\bi \right)_{\mathbf{i} \in \Gamma}  \right\|_{\ell_{\mathbf{t}} \left( \cE \right)}
\geq |\alpha_i| \left\| \left( T z_\mathbf{i} \right)_{\mathbf{i} \in \Gamma}  \right\|_{\ell_{\mathbf{t}} \left( \cE \right)} = \infty.
\]
This proves the pointwise $\fc$-lineability of $A \setminus B$.

\vspace*{3mm}

Now we turn to the pointwise $\fc$-spaceability of $A \setminus B$. The obvious candidate for the closed $\fc$-dimensional subspace of $A$ is $\overline{\Psi (\ell_p)}$. We will prove that $\overline{\Psi (\ell_p)}\setminus\{0\} \subset A\setminus B$. Let us fix $S = \lim_k \Psi_k \in \overline{\Psi (\ell_p)} \setminus \{0\}$, with $\Psi_k := \Psi_{a^{(k)}}$ and  $a^{(k)} = \left( \alpha_j^{(k)} \right)_j \in \ell_p$ for all $k \in \mathbb{N}$.

First let us show that there exists the limit $ \lim_k \alpha_1^{(k)}$. Notice that, for all $b = (\beta_n)_n \in \ell_p,\, x \in E_1 \times \cdots \times E_m$ and $j \in \mathbb{N}_1$, $\Psi_b x (j) = \beta_1 Tx(j)$, taking $x_0 \in E_1 \times \cdots \times E_m$ and $j_0 \in \mathbb{N}_1$ with $Tx_0(j_0) \neq 0$, we get
\[
S x_0 (j_0)
= \lim_{k \to \infty} \Psi_k x_0 (j_0)
=\left( \lim_{k \to \infty} \alpha_1^{(k)}\right) Tx_0 (j_0),
\]
and this guarantees the existence of the aforementioned limit.

Now let $z^{(j)} \in \ell_{u_j}^w(E_j),\, j=1,\ldots,m$, as in \eqref{N1-norm-infty}. Suppose that $\lim_k \alpha_1^{(k)} \neq 0$. Notice that, for all $x \in E_1 \times \cdots \times E_m$ and all $a \in \ell_p$,
\begin{align*}
\left\| \Psi_a  x \right\|_{\cE}
& = \left\|\left(\alpha_1 Tx(i)+\sum_{j\geq 2}\alpha_j T_jx(i)\right)_{i\in \mathbb{N}}\right\|_{\cE}\\
& \geq \left\|\left(\alpha_1 Tx(i)+\sum_{j\geq 2}\alpha_j T_jx(i)\right)_{i\in \mathbb{N}_1}\right\|_{\cE}\\
& = |\alpha_1| \left\| \left( T x (i) \right)_{i \in \mathbb{N}_1}  \right\|_{\cE}. 
\end{align*}
Then
\[
\left\| S z_\bi \right\|_{\cE}
\geq \lim_k |\alpha_1^{(k)}| \left\| \left( T z_\bi (i) \right)_{i \in \mathbb{N}_1}  \right\|_{\cE} 
\]
yields
\[
\left\| \left( S z_\bi \right)_{\mathbf{i} \in \Gamma}  \right\|_{\ell_{\mathbf{t}} \left( \cE \right)}
\geq \lim_k |\alpha_1^{(k)}| \cdot \left\| \left( T z_\mathbf{i} \right)_{\mathbf{i} \in \Gamma}  \right\|_{\ell_{\mathbf{t}} \left( \cE \right)} = \infty.
\]

Let us turn to the case $\lim_k \alpha_1^{(k)} = 0$. Since $S\neq 0$, there are  $x\in E_1 \times \cdots \times E_m$ and $i_0\in\mathbb{N}$ such that $Sx(i_0)\neq 0$. There are unique $l,t\in \mathbb{N}$ such that $i_0=n_t^{(l)}$, thus
\begin{align*}
	Sx(i_0)&=\lim_{k \to \infty} \Psi_k x (i_0)\\
	& = \lim_{k \to \infty} \left(\alpha_1^{(k)} Tx (i_0)+\sum_{j\geq 2}\alpha_j^{(k)}T_jx(i_0)\right)\\
	& = \lim_{k \to \infty}\left(\alpha_1^{(k)} Tx (n_t^{(l)})+\alpha_l^{(k)}T_lx(n_t^{(l)})\right)\\
	&  = \lim_{k \to \infty} \alpha_l^{(k)}Tx(t),
\end{align*}
and, hence $\lim_{k \to \infty} \alpha_l^{(k)}\neq 0$. Moreover,
\begin{align*}
\|Sx\|_{\cE}& =\left\|\lim_{k \to \infty}\left[\alpha_1^{(k)}Tx+\sum_{j\geq2}\alpha_j^{(k)} T_j x\right]\right\|_\cE\\
& = \lim_{k \to \infty}\left\|\sum_{j\geq2}\alpha_j^{(k)} T_j x\right\|_\cE\\
& \geq \lim_{k \to \infty}\left\|\left(\sum_{j\geq2}\alpha_j^{(k)} T_j x(i)\right)_{i\in \mathbb{N}_{l}}\right\|_\cE\\
& = \lim_{k \to \infty}\left\|\left(\alpha_{l}^{(k)} T_{l} x(i)\right)_{i\in \mathbb{N}_{l}}\right\|_\cE\\
& = \lim_{k \to \infty} \left|\alpha_{l}^{(k)}\right|\left\|\left(T_{l} x(i)\right)_{i\in \mathbb{N}_{l}}\right\|_\cE\\
& = \lim_{k \to \infty} \left|\alpha_{l}^{(k)}\right|\left\|\left(Tx(i)\right)_{i\in \mathbb{N}}\right\|_\cE\\
&= \lim_{k \to \infty} \left|\alpha_{l}^{(k)}\right| \|T x\|_\cE
\end{align*}
for all $ x \in E_1 \times \cdots \times E_m$. In particular, for such $z_\bi$ as in \eqref{N1-norm-infty},
\[
\left\| \left( S z_\bi \right)_{\mathbf{i} \in \Gamma}  \right\|_{\ell_{\mathbf{t}} \left( \cE \right)}
\geq \left( \lim_{k \to \infty} \left|\alpha_{l}^{(k)}\right| \right) \cdot
\left\| \left( T z_\mathbf{i} \right)_{\mathbf{i} \in \Gamma}  \right\|_{\ell_{\mathbf{t}} \left( \cE \right)} = \infty,
\]
this leads to $S \notin B$ and, therefore, the proof is complete.
\end{proof}

We obtain the following result proceeding in the same manner as in the proof of Theorem \ref{th-gen}. Furthermore, taking $\Lambda = \text{Diag}\left( \mathbb{N}^m \right) := \left\{ \left( n,\cdots,n \right) \in \mathbb{N}^m : n \in \mathbb{N} \right\}$ and $\Gamma = \mathbb{N}^{m}$, we return to the classical setting of absolutely and multiple summing multilinear operators.

\begin{theorem} 
Let $\Lambda \subset \mathbb{N}^m$ be a set of indexes and $\br,\bs \in [1,+\infty)^m$. Then
\[
\mathcal{L}\left( E_1, \dots, E_m; \cE \right)
\setminus 
\Pi_{(\br,\bs)}^\Lambda \left( E_1, \dots, E_m; \cE \right)
\]
is either empty or pointwise $\mathfrak c$-spaceable.
\end{theorem}

\begin{corollary} 
Let $q \in \left( 0,\infty \right]$ and $\br,\bs,\bt,\bu \in [1,+\infty)^m$. Then each one of the sets
\[
\mathcal{L} \left(E_1,\ldots,E_m;\ell_q \right)
\setminus
\Pi_{(\bt, \bu)}^{\textrm{ms}} \left(E_1,\ldots,E_m;\ell_q \right),
\quad 
\mathcal{L} \left(E_1,\ldots,E_m;\ell_q \right)
\setminus
\Pi_{(\bt, \bu)}^{\textrm{as}} \left(E_1,\ldots,E_m;\ell_q \right)
\]
and
\[
\Pi_{(\br, \bs)}^{\textrm{as}} \left(E_1,\ldots,E_m;\ell_q \right)
\setminus
\Pi_{(\bt, \bu)}^{\textrm{ms}} \left(E_1,\ldots,E_m;\ell_q \right)
\]
is either empty or pointwise $\mathfrak c$-spaceable.
\end{corollary}

\subsection{Non-absolutely summing operators on classical sequence spaces} 

We now turn our attention to linear bounded operators on the classical quasi-Banach spaces $\ell_q$ for $0<q<1$ that fail to be absolutely summing. More precisely, for $0<q<1$, we investigate large topological linear structures within
$
\mathcal{L} \left(\ell_{q},\ell_{q}\right)
\setminus
\bigcup_{1 \leq s\leq r < \infty} \Pi_{\left(r, s\right)} \left(\ell_{q}, \ell_{q} \right).
$
This matter was recently investigated in \cite{DanielT}, where the $\mathfrak{c}$-lineability of the set mentioned earlier was obtained. We take a step forward, providing that the set is $\mathfrak{c}$-spaceable. The following result generalizes \cite[Theorem 3.1]{DanielT}.

%
%


\begin{theorem} 
Let $0<q<1$. Then
\[
\mathcal{L}\left(\ell_{q};\ell_{q}\right)
 \setminus \bigcup_{1\leq s\leq r<\infty} \Pi_{\left(  r,s\right)} \left( \ell_{q};\ell_{q}\right)
\]
is $\mathfrak c$-spaceable. Moreover, the result is sharp, since it is not valid for $q\geq1$.
\end{theorem}

\begin{proof}
The starting point is a key result due to Maddox, who proved that the identity map on $\ell_p$ is never absolutely summing. More precisely, if $0<q<1$ and $1\leq s\leq r<\infty$, then the identity map $Id:\ell_{q} \to \ell_{q}$ is not $\left(r, s\right)$-absolutely summing (for more details see \cite{DanielT}). Hence we can take $T:=Id \in \mathcal{L}\left(\ell_{q};\ell_{q}\right) \setminus \bigcup_{1\leq s\leq r<\infty} \Pi_{\left(  r,s\right)} \left( \ell_{q};\ell_{q}\right)$. The argument is similar to the proof of Theorem \ref{th-gen}. We mention just the slight changes needed. In Proposition \ref{prop_pw-mothervector} we can take any countably disjoint decomposition $\left( \mathbb{N}_k \right)_{k \in \mathbb{N}}$; and the operator $\Psi$ in item (ii) must be defined as $\Psi_a := \sum_{j=1}^{\infty} \alpha_j T_j$ for $ a=(\alpha_j)_{j\in \mathbb{N}} \in \ell_p$. The last statement was observed in \cite{DanielT} as for $q \geq 1$ the coincidence $\Pi_{\max\{2,q\},1} \left(\ell_q;\ell_q\right) = \mathcal{L} \left(\ell_q;\ell_q\right)$ is well known.
\end{proof}

\section{Absolutely summing and Dunford-Pettis operators}  \label{sec-dp}

It is well known that the class of Dunford-Pettis operators plays an important role in general Banach theory and Operator Theory. An operator $T : E \to F$ is called \emph{Dunford-Pettis} (or \emph{completely continuous}) if $Tx_n$ converges to $Tx$ (in $F$), whenever $x_n$ converges weakly to $x$ in $E$. From now on the class of Dunford-Pettis operators from Banach spaces $E$ into $F$ will be denoted by $DP(E,F)$.
Observe that a compact operator must be a \emph{Dunford-Pettis operator} and, clearly, the two notions coincide when $E$ is reflexive. The interested reader can find more information and a panorama on the subject in the work \cite{diestelsurvey} and the references therein.

Bennet, in \cite{Bennet}, showed a connection between absolutely $\left(r,s\right)$-summing and Dunford-Pettis operators: for any Banach spaces $E$ and $F$, there exists an absolutely $\left(r,s\right)$-summing operator $T:E \to F$, with $1\leq r<s<\infty$, that does not satisfy the \textit{Dunford-Pettis} property. With this we have a suitable environment to investigate exotic properties such as lineability and spaceability. In fact, the next result shows that the set of absolutely $(r,s)$-summing operators with values on $\ell_q$ that fails to be Dunford-Pettis is pointwise $\fc$-spaceable.

\begin{theorem}
Let $1\leq q \leq \infty$ and $1 \leq r< s < \infty$. Then
\[
\Pi_{(r,s)}(E,\ell_q) \setminus DP(E,\ell_q)
\]
is either empty or pointwise $\mathfrak c$-spaceable.
\end{theorem}
\begin{proof}
To shorten notation we write $A := \prod_{(r,s)}(E;\ell_q)$ and $DP := DP(E,\ell_q)$. We fix a non-trivial operator $T$ in $A \setminus DP$ and take a sequence $\left( x_n \right)_n$ be such that it weakly converges to $x\in E$ but $Tx_n$ does not converge to $Tx$ in $\ell_q$. Also let $x_0\in E$ and $j_0\in \mathbb{N}$ such that $Tx_0(j_0)\neq 0$, that is, the $j_0$-th coordinate of the sequence $Tx_0 \in \ell_q$ is non-zero. Then we can choose $\epsilon_0>0$ and $\mathbb{N}_1 \subset \mathbb{N}$ with $j_0 \in \mathbb{N}_1,\ \text{card\,} \mathbb{N}_1 = \text{card} (\mathbb{N} \setminus \mathbb{N}_1)= \aleph_0$ and such that
\begin{equation}\label{eq_DP_infinito}
\left( \sum_{l\in\mathbb{N}_1} |Tx_n(l)-Tx(l)|^q \right)^{\frac1q} \geq \epsilon_0,
\end{equation}
for some $n$ large enough.

We proceed as in the proof of Theorem \ref{th-gen}, thus we omit some details. Proposition \ref{prop_pw-mothervector} is used with $\mathcal{M} = \Pi_{(r,s)}$. Fixed a countably disjoint decomposition $\mathbb{N} \setminus \mathbb{N}_1 = \bigcup_{j\geq2} \mathbb{N}_j$, the sequence of operators $(T_k)_{k\in\mathbb{N}}$ from Proposition \ref{prop_pw-mothervector} (i) fulfills $\left\|T_{k}\right\|_A = \|T\|_A$ and $T_k \in A \setminus DP$ for all $k$. Since $A$ is Banach, in Proposition \ref{prop_pw-mothervector} (ii) we have $p=1$ and the map constructed, $\Psi :\ell_1 \to A$, is such that $\Psi (\ell_p) \subset A$ is a $\fc$-dimensional space containing $T$. We just need to prove that $\Psi_a \notin DP$ for all non-zero $a=(\alpha_j)_{j=1}^{\infty} \in \ell_1$.
\vspace*{3mm}

First, suppose that $\alpha_1= 0$. Since $\{T_k x \}_{k\in\mathbb{N}}$, are sequences in $\ell_q$ with supports pairwise disjoint for all $x$,
\begin{align*}
\sum_{l=1}^{\infty}\left\vert \Psi_a  x_{n}(l)-\Psi_a x(l)\right\vert
^q 
& =\sum_{l=1}^{\infty}\left\vert \sum_{j=2}^{\infty
}\alpha_{j}T_{j}x_{n}(l)-\alpha_{j}T_{j}x(l) \right\vert^{q}
& =\sum_{l=1}^{\infty}\left\vert \sum_{j=2}^{\infty
}\alpha_{j}T_{j}(x_{n}-x)(l) \right\vert^{q}.
\end{align*}
For each $l\in \mathbb{N}$, there are unique $t, s \in \mathbb{N}$ such that $l=n_t^{(s)}$, hence
\begin{align*}
\sum_{l=1}^{\infty}\left\vert \sum_{j=1}^{\infty
}\alpha_{j}T_{j}(x_{n}-x)(l) \right\vert^q
&=\sum_{s,t=1}^{\infty}\left| \sum_{j=1}^{\infty
}\alpha_{j}T_{j}(x_{n}-x)\left(n_t^{(s)}\right)\right|^q \\
&= \sum_{s,t=1}^{\infty}\left|\alpha_{s}T_{s}(x_{n}-x)\left(n_t^{(s)}\right)\right|^q\\
&= \sum_{s=1}^{\infty} \left|\alpha_s\right|^q \sum_{t=1}^{\infty}\left|T(x_n-x)(t)\right|^q\\
&\geq \|a\|_{\ell_q}^q \cdot \epsilon_0^q,
\end{align*}
where the last inequality follows by \eqref{eq_DP_infinito}. On the other hand, $\alpha_1\neq0$ implies
\begin{align*}
&\left\|\Psi_a  x_n-\Psi_a  x\right\|_{\ell_q}^q \\
&\geq \sum_{l\in\mathbb{N}_1} \left| \alpha_1 T(x_n-x)(l) + \sum_{j=2}^{\infty} \alpha_j T_j(x_n-x)(l)\right|^q \\
&=|\alpha_1|^q \cdot  \sum_{l\in\mathbb{N}_1}\left|T(x_n-x)(l)\right|^q \\
&\geq |\alpha_1|^q \cdot \epsilon_0^q.
\end{align*}
Again, the last inequality follows by \eqref{eq_DP_infinito}. This yields $\Psi_a \notin DP$ and, therefore, the pointwise $\fc$-lineability of $A\setminus DP$.

\vspace*{3mm}

Now we prove that $\overline{\Psi(\ell_1)} \subset \left(A \setminus DP \right) \cup \{0\}$. Let $S = \lim_k \Psi a^{(k)} \in \overline{\Psi (\ell_1)} \setminus \{0\}$, with 
$a^{(k)} = \left( \alpha_j^{(k)} \right)_j \in \ell_1$ for all $k \in \mathbb{N}$. It remains for us to verify that $S \notin DP$. Let us consider  $\epsilon_0>0$ and $ (x_n)_{n\in \mathbb{N}},\, x \in E$ as in \eqref{eq_DP_infinito}. By the same reasoning from the proof of Theorem \ref{th-gen}, there exists $\lim_k\alpha_1^{(k)}$ and first we suppose that this limit is non-zero. Thus we have
\begin{align*}
\|S(x_n-x)\|_{\ell_q}^q
&= \left\|\lim_{k\to \infty} \left[ \alpha_1^{(k)} T(x_n-x) + \sum_{j=2}^{\infty} \alpha_j^{(k)} T_j(x_n-x) \right] \right\|_{\ell_q}^q\\
&=\sum_{l\in\mathbb{N}} \lim_{k\to \infty} \left|\alpha_1^{(k)} T(x_n-x)(l) + \sum_{j=2}^{\infty} \alpha_j^{(k)} T_j(x_n-x)(l)\right|^q\\
&\geq \sum_{l\in\mathbb{N}_1} \lim_{k\to \infty} \left| \alpha_1^{(k)} T(x_n-x)(l) + \sum_{j=2}^{\infty} \alpha_j^{(k)} T_j(x_n-x)(l) \right|^q\\
&= \sum_{l\in\mathbb{N}_1}\lim_{k\to \infty}\left|\alpha_1^{(k)}T(x_n-x)(l)\right|^q\\
&= \lim_{k \to \infty} |\alpha_1^{(k)}|^q \cdot \sum_{l\in\mathbb{N}_1}\left|T(x_n-x)(l)\right|^q\\
&\geq \lim_{k \to \infty} |\alpha_1^{(k)}|^q \cdot \epsilon_0^q.
\end{align*}
The argument of the case $\lim_k\alpha_1^{(k)} = 0$ is similar. Hence $S$ fails the be a Dunford-Petis operator, and this concludes the proof.
\end{proof}

\section{Summing operators on general summable scalar families} \label{sec-tuples}

In this section we deal with operators taking values on $\ell_q(I)$, the space of $q$-summable scalar family indexed over a fixed non-empty set $I$, and $q\in (0,\infty)$. Recall that $\ell_q(I)$ is the vector space of all functions $f: I \to \mathbb{K}$ such that $\sum_{i \in I} |f(i)|^q < \infty$, with this sum defined by
\[
\sum_{i \in I}\vert f(i)\vert^q
:= \sup \left\{ \sum_{i\in F}\vert f(i)\vert^q: F \textrm{ is a finite subset of } I\right\},
\]
and also recall that $\ell_q(I)$ endowed with
$
\| f\|_q = \left(\sum_{i\in I}\vert f(i)\vert^q\right)^{1/q}
$
is a quasi-Banach space. For details about the space $\ell_q(I)$ and related results we refer to \cite{pietsch-book}.

In this setting we investigate the \emph{geometric} notion of $(\alpha,\beta)$-lineability, recently introduced by F\'avaro, Pellegrino and Tomaz in \cite{favaro}. More precisely, let $\alpha,\beta,\lambda$ be cardinal numbers and $V$ be a vector space, with $\dim V=\lambda$ and $\alpha<\beta\leq\lambda$. A set $A\subset V$ is \emph{$\left(\alpha, \beta\right)$-lineable} if it is $\alpha$-lineable and for every subspace $W_{\alpha}\subset V$ with $W_{\alpha}\subset A\cup\left\{0\right\}  $ and $\dim W_{\alpha}=\alpha$, there is a subspace $W_{\beta}\subset V$ with $\dim W_{\beta}=\beta$ and $W_{\alpha}\subset W_{\beta}\subset A\cup\left\{  0\right\}$. When $V$ is a topological vector space and the subspace $W_\beta$ can be chosen closed, $A$ is said to be \emph{$\left(\alpha, \beta\right)$-spaceable}. Observe that this encompasses the lineability notion (take $\alpha=0$) and it is clear that pointwise $\alpha$-lineability / spaceability implies $(1,\alpha)$-lineability / spaceability. Recent papers provided results in this fashion \cite{digo-anselmo-bbms,dfpr_racsam2020,FPP,pr_pointw}. Next, we obtain a result of this type.

\begin{proposition}
Let $q\in (0,\infty]$, $I$ be a non-empty set, and $E_i$ be infinite-dimensional Banach spaces with $\dim E_i < \text{card\,}I$ for $i=1,\dots,m$. Also consider $\br,\bs,\bt,\bu \in [1,+\infty)^m$ and $\Lambda \subset \Gamma \subset \mathbb{N}^m$ sets of indexes. Then
\[
\Pi_{(\br, \bs)}^{\Lambda}  \left( E_1,\ldots,E_m; \ell_q\left(I \right) \right) \setminus \Pi_{(\bt, \bu)}^{\Gamma} \left( E_1,\ldots,E_m; \ell_q\left(I \right) \right)
\]
is either empty or $(\alpha,\text{card } I)$-lineable for a cardinal $\alpha$ with $\alpha<\text{card\,} I$.
\end{proposition}
\begin{proof}
We will write $A=\Pi_{(\br, \bs)}^{\Lambda}  \left( E_1,\ldots,E_m; \ell_q\left(I \right) \right)$ and $B=\Pi_{(\bt, \bu)}^{\Gamma} \left( E_1,\ldots,E_m; \ell_q\left(I \right) \right)$ to simplify the notation. Suppose there is $T \in A\setminus B$ and write $I=\bigcup_{j\in I} I_j$ as a pairwise disjoint union, with $\text{card } I_j = \text{card } I$ for all $j\in I$. For each $j\in I$, we denote $I_j := \left\{\eta_i^{(j)}:i\in I\right\}$, and thus $I = \left\{\eta_i^{(j)}:i,j\in I\right\}$. Then for any $\nu \in I$, there are unique $i,j\in I$ such that $\nu = \eta_i^{(j)}$. Similarly to the construction in Proposition \ref{prop_pw-mothervector} (i), for all $\nu \in I$ we define $T_\nu : E_1 \times\cdots\times E_m \to \ell_q(I)$ by
\begin{equation*}
T_\nu x\left(\eta_i^{(j)}\right) = 
\begin{cases}
0, & \mbox{if $j\neq \nu$}\\
Tx(i), & \mbox{if $j=\nu$}.
\end{cases} 
\end{equation*}
Thus $\|T_\nu\|_A = \|T\|_A$ and this implies $T_\nu \in A\setminus B$ for all $\nu \in I$. Moreover, $\{T_\nu : \nu \in I\}$ is linearly independent and
$
X := \text{span\,} \{T_\nu : \nu \in I\} \subset \left( A \setminus B\right) \cup \{0\}.
$
Notice that, since $\dim X = \text{card\,} I$, this yields the $\text{card\,}I$-lineability of $A\setminus B$.

Now we turn to proving that $A\setminus B$ is $(\alpha,\text{card\,} I)$-lineable for all $\alpha<\text{card } I$. Let $\mathcal{V} \subset A\setminus B \cup \{0\}$ be an arbitrary $\alpha$-dimensional subspace. Consider
\[
\mathcal{I} := 
\{Vx : V\in \mathcal{V} \text{ and } x\in E_1\times\cdots\times E_m\}
\subset \ell_q(I).
\]
and also denote by $\Delta$ the set of indexes formed by all non-zero coordinates of each scalar family $Vx \in \mathcal{I} \subset \ell_q(I)$. Note that the number of such coordinates is not bigger than
\[
\beta := \text{card\,} (E_1 \times\cdots\times E_m) \cdot \aleph_0 \cdot \text{card\,} \mathcal{V} <\text{card\,} I.
\]

Consider a new disjoint decomposition of $I$, 
\[
I=\left(\bigcup_{j\in I}\widetilde{I}_j\right)\cup \Delta,
\]
with $\text{card\,} \widetilde{I}_j = \text{card\,} I$ for all $j\in I$. As before, we have to construct operators $F_\nu : E_1\times\cdots\times E_m \to \ell_q(I)$ such that the support of the scalar family $F_\nu x \in \ell_q(I)$ lies in $\widetilde{I}_\nu$, and also
$
\|F_\nu x\|_{\ell_p(I)} =\|T_\nu x \|_{\ell_q(I)},
$
for all $x\in E_1\times\cdots\times E_m$. Hence $F_\nu$ belongs to $A\setminus B$ for all $\nu \in I$ and the subspace
\[
\mathcal{W} := \text{span\,}\{F_\nu,V: \nu \in I \text{ and } V \in \mathcal{V}\}
\]
fulfills $\mathcal{V}\subset \mathcal{W}\subset (A\setminus B)\cup \{0\}$ and $\dim \mathcal{W}=\text{card\,} I$.
\end{proof}

\end{document}